\documentclass[10pt]{amsart}
\usepackage{tikz,array, verbatim}
\usepackage{amsfonts, amsmath, latexsym, epsfig, caption}
\usepackage{amssymb, color}
\usepackage{epsf}
\usepackage{array}
\usepackage{ragged2e}
\usepackage{hyperref}
\DeclareMathOperator{\GL}{GL}

\DeclareMathOperator{\AGL}{AGL}
\DeclareMathOperator{\Sym}{Sym}

\DeclareMathOperator{\tconv}{conv}
\DeclareMathOperator{\vol}{vol}
\DeclareMathOperator{\tvert}{vert}

\title{Periodic triangulations of $\ZZ^n$}

\def\QuotS#1#2{\leavevmode\kern-.0em\raise.2ex\hbox{$#1$}\kern-.1em/\kern-.1em\lower.25ex\hbox{$#2$}}

\begin{document}

\author[M. D. Sikiri\'c]{Mathieu Dutour Sikiri\'c}
\address{Mathieu Dutour Sikiri\'c, Rudjer Boskovi\'c Institute, Bijeni\v{c}ka 54, 10000 Zagreb, Croatia}
\email{mathieu.dutour@gmail.com}

\author[A. Garber]{Alexey Garber}
\address{Alexey Garber, School of Mathematical \& Statistical Sciences, The University of Texas Rio Grande Valley, 1 West University Blvd, Brownsville, TX, 78520, USA}
\email{alexeygarber@gmail.com}

\thanks{The first author has been supported by the Humboldt Foundation}

\newcommand{\RR}{\ensuremath{\mathbb{R}}}
\newcommand{\NN}{\ensuremath{\mathbb{N}}}
\newcommand{\QQ}{\ensuremath{\mathbb{Q}}}
\newcommand{\CC}{\ensuremath{\mathbb{C}}}
\newcommand{\ZZ}{\ensuremath{\mathbb{Z}}}
\newcommand{\TT}{\ensuremath{\mathbb{T}}}

\newtheorem{theorem}{Theorem}[section]
\newtheorem{proposition}[theorem]{Proposition}
\newtheorem{corollary}[theorem]{Corollary}
\newtheorem{lemma}[theorem]{Lemma}
\newtheorem{problem}[theorem]{Problem}
\newtheorem{conjecture}{Conjecture}
\newtheorem{question}{Question}
\newtheorem{claim}{Claim}
\newtheorem{remark}[theorem]{Remark}
\theoremstyle{definition}
\newtheorem{definition}[theorem]{Definition}

\begin{abstract}
We consider in this work triangulations of $\ZZ^n$ that are periodic along $\ZZ^n$.
They generalize the triangulations obtained from Delaunay tessellations of lattices.
Other important property is the regularity and central-symmetry property of
triangulations.

Full enumeration for dimension at most $4$ is obtained.
In dimension $5$ several new phenomena happen:
there are centrally-symmetric triangulations that are not Delaunay,
there are non-regular triangulations (it could happen in dimension $4$) and
a given simplex has a priori an infinity of possible adjacent simplices.
We found $950$ periodic triangulations in dimension $5$ but finiteness is unknown.

\end{abstract}

\maketitle

\section{Introduction}

Given a positive definite quadratic form $A$ we obtain a tessellation of $\ZZ^n$
by taking the projection of the facets of the convex hull of
$\left\{(x, x^T A x)\mbox{~for~} x\in \ZZ^n\right\}$.
This triangulation is $\ZZ^n$-periodic, centrally symmetric and is called the Delaunay
tessellation \cite{Edelsbrunner_book}.
For dimension at most $5$ those tessellations are classified and there
are $1$, $2$, $5$, $52$ and $110244$ for $1\leq n\leq 5$ \cite{stogrin-1973,ClassifDim5}
up to the action of $\GL_n(\ZZ)$. If one limit oneself to the Delaunay triangulations
formed of simplices then the number of types is $1$, $1$, $1$, $3$ and $222$
\cite{br-1973,rb-1976,eg-2002} for $1\leq n\leq 5$, respectively.
For $n=6$ Baburin and Engel \cite{BE} reported more than 500'000'000 non-equivalent
triangulations.
A triangulation is called {\em regular} if it is obtained as projection of facets
of an infinite convex body of vertices $(x, f(x))$ for $f$ a function defined on $\ZZ^n$.
This generalizes the Delaunay tessellations.

In this paper we consider general triangulations of the point set $\ZZ^n$ which are
invariant under translations by $\ZZ^n$ and are face-to-face.
Such triangulations can be viewed as decomposition of a torus into a cell complex
with one vertex where all cells are simplices.
Other triangulations of the torus into simplices were considered in \cite{GrigisDelaunay,ItohAcute}.
In Section \ref{sec:preliminaries} we consider general results on periodic triangulations of $\ZZ^n$
in particular on groups, simplices and refinement of periodic tilings.
In Section \ref{Sec_Comput_Tools} we detail a number of computational tools for testing Delaunayness
and regularity that we use in this work.

In Section \ref{sec:3dim} we prove that for $n\leq 3$
all such triangulations are Delaunay. For $n=4$ a non-centrally symmetric triangulation
named ``red-triangular'' was described in \cite[Example 5.13.1]{Alexeev_semiabelian}.
We also prove in Section \ref{sec:4dim} that this triangulation together with the Delaunay ones form all the set
of triangulations up to the action of $\GL_4(\ZZ)$.

In dimension $n\geq 5$ full enumeration of periodic triangulation appears to be difficult.
First of all the finiteness is not proved and may not hold since we prove in
Section \ref{sec:dim5} that given a simplex of volume $1$ there are a priori
infinitely many possibilities for an adjacent simplex.

In Section \ref{sect:flip} we obtain $950$ non-isomorphic periodic triangulations of $\ZZ^5$
but we do not know if this is the complete list. This list allow us to prove that there are
centrally symmetric but not Delaunay triangulations and non-regular triangulations.

In \ref{sec:open_questions} we list several open questions on enumeration, extensibility and regularity
properties of periodic triangulations of $\ZZ^n$ that should be of general interest.





\section{General properties of periodic triangulations}\label{sec:preliminaries}

\begin{definition}\label{DEF_kdim_lattice}
A {\em $k$-dimensional lattice} is a discrete subgroup of $\RR^n$ of
rank $k$, i.e. a set of the form $L=\ZZ v_1 + \dots + \ZZ v_k$ with
linearly independent vectors $v_1,\ldots, v_k$.
\end{definition}

Throughout the paper we will work with the lattice $L=\ZZ^n$. The group of affine transformations
preserving a $n$-dimensional lattice is isomorphic to $\AGL_n(\ZZ)$.

\begin{definition}\label{DEF_triangulation}
A {\em partial triangulation} $\mathcal{PT}$ of $\ZZ^n$ is a packing of $\RR^n$ by $n$-dimensional simplices with the vertex set $\ZZ^n$, i.e. representation of $\RR^n$ as a union of countably many simplices with integer vertices such that the intersection of any pair of simplices is a face of both. 

A {\em triangulation} ${\mathcal T}$ is a partial triangulation which is also a tiling.

Triangulation $\mathcal{T}$ is called {\em $\ZZ^n$-periodic}, or just {\em periodic} if $\mathcal{T}+v=\mathcal{T}$ for every $v\in\ZZ^n$.
\end{definition}

\begin{definition}\label{DEF_symmetry_group}
The {\em symmetry group} $\Sym({\mathcal T})$ of a periodic triangulation ${\mathcal T}$ of $\ZZ^n$ is the group of affine transformations of $\RR^n$ preserving $\mathcal T$. 

The group $\Sym({\mathcal T})$ contains $\ZZ^n$ as a normal subgroup of finite index. The quotient $\Sym({\mathcal T}) / \ZZ^n$ is called the {\em point group} $Pt({\mathcal T})$. 

The symmetry group is {\em split} if $\Sym({\mathcal T})$ is a semi-direct product $\ZZ^n \rtimes Pt({\mathcal T})$. This is equivalent to having $Pt({\mathcal T})$ being realized as a subgroup of $\Sym({\mathcal T})$.
\end{definition}



\begin{proposition}
The symmetry group of a periodic triangulation ${\mathcal T}$ is split.
\end{proposition}
\begin{proof} Let $v_0=0$, $v_1=e_1$, \dots, $v_n=e_n$ with $(e_i)$ the standard basis of $\ZZ^n$. Let $f$ be a symmetry of the lattice.
Define $v'_i = f(v_i)$ for $0\leq i\leq n$ and write the transformation $f$ in matrix form as $Ax + b$.
Then 
\begin{equation*}
  \left(\begin{array}{ccc}
    1    & \dots & 1 \\
    v'_0 & \dots & v'_{n}
  \end{array}\right)
  =
  \left(\begin{array}{cc}
    1 & 0\\
    b & A
  \end{array}\right)
  \left(\begin{array}{ccc}
    1   & \dots & 1 \\
    v_0 & \dots & v_{n}
  \end{array}\right)
\end{equation*}
This equation implies that $A$ and $b$ are integral.
Thus $f$ is the composition of a transformation preserving the origin and an integral translation.
This implies that the symmetry group is split. \end{proof}

\begin{definition}\label{DEF_volume_simplex}
Let $\Lambda$ be a $d$-dimensional lattice with fundamental volume $V$, and let $S$ be a $d$-dimensional simplex with vertices from $\Lambda$. In this case volume of $S$ is $k\cdot\frac{V}{d!}$ for some integer $k$, and we will say that {\it relative volume} of $S$ is $k$.

In the following we will refer to relative volume of $S$ as just volume of $S$, or $\vol(S)$, unless we need to emphasize the dimension.
\end{definition}

\begin{proposition}\label{prop:volume}
Let $S$ be a simplex of an $n$-dimensional periodic triangulation. 

(i) We have the inequality $\vol(S) \leq n!$ .

(ii) If the triangulation is centrally symmetric then
\begin{equation*}
\vol(S) \leq \frac{2^n}{ {2n \choose n} } n! .
\end{equation*}
\end{proposition}
\begin{proof}
If the periodic triangulation is formed by the simplices $S_1$, \dots, $S_p$ and
their $\ZZ^n$ translations then we have the equality
\begin{equation*}
\sum_{i=1}^p \vol(S_i) = n! .
\end{equation*}
from which (i) obviously follows.

The proof of (ii) follows exactly the same arguments as \cite[Proposition 14.2.4]{DL}. \end{proof}

\begin{definition}\label{DEF_ScalP_extension}
Let ${\mathcal T}$ be a periodic tiling of $\ZZ^n$ by polytopes having only integer points as vertices.
Let $A$ be a positive definite quadratic form on $\ZZ^n$.
Then $A$ induces another tiling $Ref_A({\mathcal T})$ of $\ZZ^n$ defined
on each polytope $P\in T$ as projection of the lower facets of 
\begin{equation*}
Scal(P) = \tconv \left\lbrace (x, A[x])\mbox{~for~} x\in \tvert(P)\right\rbrace.
\end{equation*}
\end{definition}

\begin{lemma}\label{Refinement_Theorem}
For a $n$-dimensional periodic tiling ${\mathcal T}$ and a positive definite quadratic form $A$
the following properties hold:

(i) $Ref_A({\mathcal T})$ is a periodic tiling of $\ZZ^n$ which is a refinement of ${\mathcal T}$ .

(ii) If $g\in \GL_n(\ZZ)$ preserves $A$ and belongs to the point group of ${\mathcal T}$ then it belongs to the point group of $Ref_A({\mathcal T})$.

(iii) If $A$ is generic then $Ref_A({\mathcal T})$ is a triangulation
\end{lemma}
\begin{proof}
(i) Let us consider for each polytope $P$ of ${\mathcal T}$ the scaling map used to describe the Delaunay polytopes.
\begin{equation*}
Scal(P) = \tconv \left\lbrace (x, A[x])\mbox{~for~} x\in \tvert(P)\right\rbrace
\end{equation*}
The lower facets of $Scal(P)$ define a tiling of $P$ into polytopes.
It also defines a tiling of the faces of $P$. If a face $F$ of ${\mathcal T}$
is contained into polytopes $P_1$, \dots, $P_m$ then the induced tilings are compatible.
The tiling is periodic since $A[x]$ differ on two different translate of a tile by an affine
term.

(ii) is trivial.

(iii) If $A$ is generic then the tiling induced by $Scal(P)$ is a triangulation which proves the claim.
\end{proof}


\begin{definition}\label{DEF_regularity}
A triangulation $T$ of $\ZZ^n$ is called {\em regular} if there exists
a function $f:\ZZ^n \rightarrow \RR$ such that:
\begin{itemize}
\item The points $(x, f(x))$ are vertices of the convex polyhedron $H(f) = \tconv\{ (x,f(x)) | x\in \ZZ^n\}$ in $\RR^{n+1}$.
\item The simplices of $T$ are orthogonal projections of the facets of $H(f)$ onto $\RR^n$.
\end{itemize}
\end{definition}
Obviously Delaunay triangulations are regular.

\section{Computational tools}\label{Sec_Comput_Tools}

\subsection{Testing Delaunay property}\label{Sec_Comput_Tools_-_Test_Delaunay}
Given a periodic triangulation ${\mathcal T}$ of $\ZZ^n$ we can test if it is Delaunay
in the following way: We determine all facets $F$ of simplices $S$ of ${\mathcal T}$ up to translations.
Any such facet $F$ is contained in exactly two Delaunay simplices $S_1=\tconv(F \cup \{v_1\})$
and $S_2=\tconv(F \cup \{v_2\})$. We then form the inequality $N_{S_1,v_2}(A) \geq 0$ with $N_{S, v}$
being a linear form called the {\em Voronoi regulator} (see \cite{VoronoiII,GenVoronoiTheoryDSV} for details).
The polyhedral cone defined by all such inequalities is called ${\mathcal P}$.
If ${\mathcal P}$ is full dimensional then every quadratic form in the interior of ${\mathcal P}$
induces ${\mathcal T}$ as Delaunay triangulation. Otherwise, it is not a Delaunay triangulation.

\subsection{Adjacency of simplices}\label{Sec_Comput_Tools_-_Test_Adjacency_Simplices}
Suppose we have two simplices $\Delta_1$ and $\Delta_2$ and we want to check if the $\ZZ^n$
translates of $\Delta_1$ and $\Delta_2$ are a priori admissible as parts of a periodic
triangulation of $\ZZ^n$.
That is we want to check that the translates do not intersect in their interior and that the
intersection is always a face of both.
If $F$ is a facet of $\Delta_1$ represented by an inequality $f(x)\geq 0$ with an affine
function $f$ then if we have $f(\Delta_2 + v) < 0$ then there is no intersection.
That is for some facet $f$ we have
$\max_{x\in T_2} f(x) + f(v) < 0$ then there is no intersection.
So, the feasible vectors $v$ are the ones that satisfies $\max_{x\in T_2} f(x + v) \geq 0$
for all facet inequalities $f$ of $\Delta_1$.

This defines a convex body ${\mathcal C}$ and the integer points can be obtained by using
exhaustive enumeration. Then for each integer point $v\in {\mathcal C}$
we check if $\Delta_1 \cap (\Delta_2 + v)$ is $n$-dimensional or not.
If it is not then we check that the intersection is a face of both.

With this method we can find the possible simplices adjacent to a given simplex $\Delta$.
That is for a facet $F$ of $\Delta$ and a point $v\in \ZZ^n$ we consider whether the pair
$\Delta$ and $\tconv(F\cup \{v\})$ is admissible for a periodic tessellation. By iterating
over the facets and vectors $v$ we have a list of possible candidates.
However, we have no method for restricting the set of possible vectors $v$ and in Section
\ref{sec:dim5} we show that the set of such vectors can be infinite.

\subsection{Testing regularity}\label{Sec_Comput_Tools_-_Test_Regularity}
Given a periodic triangulation ${\mathcal T}$ we want to check if it is regular.
According to definition \ref{DEF_regularity} the condition that the simplices
correspond to facets of the convex body $H(f)$ translate into linear inequalities
on the values of $f(v)$ for $v\in \ZZ^n$.
Thus testing regularity is equivalent to checking if an infinite dimensional linear
program has a feasible solution.

We do not have a general method for working with infinite dimensional linear programs
and thus we cannot check regularity of triangulations easily. 
What we can do instead is prove in some cases that a periodic triangulation is not regular.
Let us take a triangulation ${\mathcal T}$ of $\ZZ^n$ and select a finite set ${\mathcal S}$ of simplices.
We can consider the function $f$ on the set of vertices ${\mathcal V}$ corresponding to
the simplices of ${\mathcal S}$.
If we have two adjacent simplices $S_1$ and $S_2$ then denote by $\phi_{S_1}$ the affine linear function
coinciding with $f$ on $S_1$. For a vertex $v$ of $S_2$ which is not in $S_1$ we must have the inequality
\begin{equation*}
  f(v) > \phi_{S_1}(v).
\end{equation*}
Now if the function $f$ does exist, and therefore $\mathcal T$ is regular, then by rescaling there exists a function $f$ on ${\mathcal V}$
satisfying
\begin{equation*}
f(v) \geq \phi_{S_1}(v) + 1.
\end{equation*}
This strengthened inequalities define a polyhedral convex body ${\mathcal Q}$.
If ${\mathcal Q}$ is proven to be empty by linear programming, then we have proved
that ${\mathcal T}$ is not regular.

\subsection{Equivalence and stabilizer}\label{Sec_Comput_Tools_-_Test_Equivalence}
For enumeration purposes, we need to be able to check that two triangulations are equivalent
and to compute the stabilizer of a triangulation. The method is simply to take one simplex in a
triangulation and to consider all ways in which it may be mapped in a simplex of another or the
same triangulation. While computationally expensive, this method is adequate for the cases that
we consider which are low-dimensional.

\section{Enumeration of periodic triangulations in dimension 3}\label{sec:3dim}

The main goal of this and the next section is to give complete enumeration of
periodic triangulations in dimensions $3$ and $4$.

Let $S$ be a simplex in a triangulation (not necessary full-dimensional),
denote by $tr(S)$ the translation class of $S$, i.e. the set of all translations of $S$ by integer vector.

The following lemma is true for arbitrary dimension.

\begin{lemma}
If a facet $F$ is common for full-dimensional simplices $S_1$ and $S_2$,
then facets in the intersection of $tr(S_1)$ and $tr(S_2)$ are only
facets in $tr(F)$.
\end{lemma}

\begin{proof}
Assume contrary, then there is one more facet $F'$ of $S_1$ which
is parallel to some facet of $S_2$. If $F$ and $F'$ are in the same copy of
$S_2$ in $tr(S_2)$, then $S_1$=$S_2$. Otherwise, facets $F$ and $F'$ have a common
ridge, so $S_2$ should have two parallel ridges which is impossible.
\end{proof}



In dimension 3, it appears to be known that there is only one periodic triangulation. See work of Alexeev, \cite[Sect. 5.13]{Alexeev_semiabelian}. Here we give another straightforward approach.

In order to obtain the full classification, we show how one can find the upper bound on the relative volume of a simplex in a periodic triangulation. We apply it for dimensions 3 here and for dimension 4 in the following section. For dimension 2 it is clear that each simplex should have volume 1 (for example, from Pick's formula). 

Let $\mathcal{T}$ denote an arbitrary periodic triangulation of $\ZZ^n$. We will use a more careful approach compared to Proposition \ref{prop:volume} in order to show, that if the relative volume of a simplex exceeds a certain number (actually it is 1 for dimensions 3 and 4), then this simplex can not be included in a periodic triangulation, and in $\mathcal{T}$ particularly.

\begin{proposition}\label{prop:3dim=1}
If $\mathcal{T}$ is three-dimensional, then the relative volume of each three-dimensional simplex is $1$.
\end{proposition}

\begin{proof}
Let $ABCD$ be an arbitrary simplex of $\mathcal{T}$ with volume at least $2$. It is clear, that relative volume of every facet of $ABCD$ in corresponding sublattice is 1. So, we can choose a coordinate system (with matrix transformation from $\GL_3(\mathbb{Z})$) such that vertices of $ABCD$ will have coordinates represented by columns of the following matrix 
\begin{equation*}
\left(\begin{array}{rrrr}
0&1&0&a\\
0&0&1&b\\
0&0&0&c\\
\end{array}\right),
\end{equation*}
where $a,b,c$ are non-negative, $c\geq 2$, $a<c$, and $b<c$.

If $a$ is non-zero (similarly for non-zero $b$), then the point 
\begin{equation*}
\frac{(a-1)A+(c-a)B+D}{c}=\frac{(c-b)A + b C}{c} \pmod 1
\end{equation*}
belongs to translations of two faces of $ABCD$, namely to the two-dimensional face $ABD$ (the left-hand side of the formula) and to the edge $AC$ (the right-hand side of the formula), and is not an integer point. Thus this is a contradiction with the face-to-face property of the tiling $\mathcal{T}$.

If $a=b=0$ then we have an integer point $(0,0,1)$
in the interior of the edge $AD$ which is impossible.
\end{proof}

\begin{remark}
From this proof we can see, that each $3$-dimensional face of $\mathcal{T}$ should have the relative volume $1$, otherwise we will find a contradiction in $3$-dimensional affine space spanned by this face. Indeed, if a relative volume of a $3$-dimensional simplex is more than $1$, then according to the proof of the previous proposition, its lattice translates will intersect in a non-face-to-face manner.
\end{remark}

%
%
%

Next we establish all possible neighbors of a given simplex in a periodic triangulation $\mathcal{T}$ of $\ZZ^3$.

\begin{lemma}\label{lem:3dim-neigh}
If $n=3$, then given a simplex $S_1$ of $\mathcal{T}$ and its facet $F$, we have $3$ options for a simplex $S_2$ of $\mathcal T$ adjacent to $S_1$ by $F$. More precisely, two other facets of $S_1$ and $S_2$ must form a parallelogram.
\end{lemma}
\begin{proof}
Without loss of generality we assume, that vertices of $S_1$ are:
$A=(0,0,0)$,
$B=(1,0,0)$,
$C=(0,1,0)$,
$D=(0,0,1)$,
and $BCD$ is the common facet. The fourth vertex $E$ of $S_2$ has
coordinates $(x,y,z)$ with $x+y+z=2$. 

At least one of numbers $x,y,z$ is even, assume $z$. Then the midpoint of $EB$ has coordinates $\left(\frac{x+1}{2},\frac{y}{2},\frac{z}{2}\right)$. Among numbers $x+1$ and $y$ one is even, so this midpoint has two integer coordinates, and one half-integer. Therefore, this midpoint is a translation of one of midpoints: $AB$ or $AC$. Therefore the edge $EB$ is a translation of $AB$ or $AC$. Similarly, the edge $EC$ is a translation of $AB$ or $AC$.

There are two options remaining for point $E$: $E=(0,0,0)=A$, or $E=(1,1,0)$. In the first case simplices $S_1$ and $S_2$ coincide, which is impossible. In the second case faces $ABC$ and $EBC$ form a parallelogram.
\end{proof}

We proceed with the classification of periodic triangulations of $\ZZ^3$. We continue to use all notations of Lemma \ref{lem:3dim-neigh} and its proof.

\begin{theorem}
There is unique periodic triangulation of $\ZZ^3$ up to $\GL_3(\ZZ)$ equivalence.
\end{theorem}

\begin{proof}
From the proof of Lemma \ref{lem:3dim-neigh} we have a pair of simplices $S_1$ and $S_2$, and a parallelogram $ABEC$ (actually a unit square). With translations of this parallelogram we can tile an arbitrary plane $z=k$ for integer $k$, so any simplex of the
triangulation should be between a pair of consecutive planes parallel to $z=0$.

Currently we have six ``unpaired'' facets of tiling (i.e. facets that belong to only one simplex currently determined): $ABC$, $ABD$, $ACD$, $EBC$, $EBD$, $ECD$. No translational class of full-dimensional simplex can contain more than two of these facets, because otherwise it will have two common facets with one of simplices $S_1$ or $S_2$. The second simplex $S_3$ incident to the facet $ABC$ has the fourth vertex on the plane $z=-1$, because $S_3$ has three vertices on $z=0$, and $ABCD$ has fourth vertex on $z=1$). Similarly, the simplex $S_4$ incident to $EBC$ has the fourth vertex on the plane $z=-1$. So, $S_3$ can not have a facet which is a translation of any of five remaining ``unpaired'' facets, since four of these facets (except $EBC$) have two vertices on lower plane and one on the upper, and the facets $EBC$ is parallel to facet $ABC$. Therefore, classes $tr(S_3)$ and $tr(S_4)$ contain only facets $ABC$ and $EBC$ from these six classes.

The remaining unpaired facets are: $ABD$, $ACD$, $EBD$, $ECD$. These classes should be contained in two translational classes of simplices (we already found four classes generated by $S_1$, $S_2$, $S_3$, and $S_4$, and we must have six translational classes in total due to volume argument). No class can cover more than two, so these four facets should be divided in pairs, and each pair should belong to one translational class. Pairs can not be from one simplex $S_1$ or $S_2$. So $ABD$ should be paired with $EBD$ or $ECD$. The first case is impossible,
because the edge $BD$ can belong only to one simplex from this class, so
this class should be $tr(ABDE)$, but $ABDE$ intersects with interior of
$ABCD$.

Therefore, the class $tr(S_5)$ contains facets $ABD$ and $ECD$, and the class $tr(S_6)$ contains remaining facets $ACD$ and $EBD$.

The edges $AB$ and $EC$ are equal and parallel, so if we translate $ECD$ so that $EC$ will coincide with $AB$ (by the vector $(0,-1,0)$), we will get the second facet of the same simplex from this class. Thus, we can
reconstruct a representative $S_5$ of this class with vertices: 
$A=(0,0,0)$ (translation of $C$ by $(0,-1,0)$),
$B=(1,0,0)$ (translation of $E$ by $(0,-1,0)$),
$D=(0,0,1)$,
$F=(0,-1,1)$ (translation of $D$ by $(0,-1,0)$).

Similarly as $S_6$ we can take the simplex with vertices:
$A=(0,0,0)$ (translation of $B$ by $(-1,0,0)$),
$C=(0,1,0)$ (translation of $E$ by $(-1,0,0)$),
$D=(0,0,1)$,
$G=(-1,0,1)$, (translation of $D$ by $(-1,0,0)$).

From four completely defined classes $tr(S_1)$, $tr(S_2)$, $tr(S_5)$, and $tr(S_6)$ we have the following facets that do not belong to the second full-dimensional simplex so far: $ABC$, $EBC$, $ADF$, $BDF$, $ADG$, $CDG$. No class $tr(S_3)$ or $tr(S_4)$ can cover more than three of these facets, otherwise it will cover two facets from one simplex (facets $ABC$ and $EBC$ cannot be covered simultaneously because they are parallel). So each class covers exactly three.

We apply Lemma \ref{lem:3dim-neigh} for simplices $ABCD$ and $S_3$ with common facet $ABC$. We have three options for the fourth vertex of $S_3$ (this vertex forms a
parallelogram with three vertices of $ABCD$):
$H_1=(1,1,-1)$ (parallelogram with $BCD$),
$H_2=(1,0,-1)$ (parallelogram with $ABD$),
$H_3=(0,1,-1)$ (parallelogram with $ACD$).

Assume that $S_3=ABCH_2$ ($S_3=ABCH_3$ is similar). It contains facets $ABC$
and $ADG$ ($ADG$ translated by $(1,0,-1)$ is $H_2BA$), but does not contain
other facets. For example, if it contains $ADF$, then lower point of $S_3$
(vertex with smallest $z$-coordinate, i.e. $H_2$) should be translated into lower
point of this facet, i.e. $A$. But this translation does not match any
facet of $S_3$ with $ADF$. Similarly with other ``unpaired'' facets, except
$EBC$, but $S_3$ already has a facet parallel to $EBC$, which is $ABC$.

So, there is only one possible case for $S_3$ which is $ABCH_1$ (the translation class contains $ABC$, $BDF$,
$CDG$). Similarly, there is only one case for $S_4=EBCH_1$ (the translation class contains $EBC$,
$ADF$, $ADG$).

We reconstructed the whole triangulation which is unique up to $\GL_3(\ZZ)$-transformation.
\end{proof}

\section{Enumeration of periodic triangulations in dimension 4}\label{sec:4dim}

As with dimension 3, we first bound the relative volume of a four-dimensional simplices.

\begin{proposition}\label{prop:4dim=1}
If $\mathcal{T}$ is a periodic triangulation of $\ZZ^4$, then volume of each four-dimensional simplex is $1$.
\end{proposition}


\begin{proof}
Let $ABCDE$ be an arbitrary simplex of $\mathcal{T}$ with volume at least $2$. We can choose a coordinate system (with matrix transformation from $\GL_4(\mathbb{Z})$) such that vertices of $ABCDE$ will have coordinates represented by columns of the following matrix 
\begin{equation*}
\left(
\begin{array}{rrrrr}
0&1&0&0&a\\
0&0&1&0&b\\
0&0&0&1&c\\
0&0&0&0&d
\end{array}
\right),
\end{equation*}
where $a,b,c,d$ are non-negative, $d\geq 2$, $a\leq b\leq c<d$.

If $c=0$, then the point $(0,0,0,1)$ lies in the interior of $AE$ which is impossible, so $c\geq 1$.

If $a+b\leq d$, then
\begin{equation*}
\frac{(c-1)A + (d-c)D + E}{d}=\frac{(d-a-b)A + a B + b C}{d} \pmod 1,
\end{equation*}
but these points lie on different faces of $ABCDE$ and the tiling will be non face-to-face.

If $b+c>d$, then
\begin{equation*}
\frac{(b+c-d-1)A + (d-b)C + (d-c)D + E}{d}=\frac{(d-a)A + a B}{d} \pmod 1
\end{equation*}
which is again a contradiction.

So, $a+b>d\geq b+c$, which contradicts with the inequality $a\leq c$.
\end{proof}

Note that the proofs of this proposition and of the similar proposition \ref{prop:3dim=1} for dimension 3 can be combined in the following corollary. 

\begin{corollary}
If $\mathcal{T}$ is an $n$-dimensional periodic triangulation, then all $3$- and $4$-dimensional faces have relative volume $1$.
\end{corollary}

This corollary allows us to formulate a local approach to enumeration of all periodic tilings. We used this approach in dimension $3$ in the previous section and now we are going to use it in dimension $4$. We can analyze local structure of the tiling $\mathcal{T}$ and show that given a simplex $S$ and its facet $F$, there are only finitely many options to attach another simplex $T$ at $F$ without violating the face-to-face property. Unfortunately this method doesn't work if dimension $n\geq 5$ as shown in Section \ref{sec:dim5}.

\begin{theorem}\label{thm:4dim-neigh}
For a fixed four-dimensional simplex $S$ and its facet $F$, there are at most 10 options for another simplex adjacent to $S$ by $F$ in a periodic triangulation of $\ZZ^4$.
\end{theorem}

\begin{proof}
We already know, that all simplices have volume 1. We fix one simplex $S_1$ with vertices
$A=(0,0,0,0)$,
$B=(1,0,0,0)$,
$C=(0,1,0,0)$,
$D=(0,0,1,0)$,
$E=(0,0,0,1)$,
and find all possibilities for the vertex $F$ of the simplex $S_2=BCDEF$ adjacent to $S_1$ by facet $BCDE$. We know that $F$ has coordinates $(x,y,z,t)$ with $x+y+z+t=2$. We will show that there are only 10 options for the vertex $F$.

We will do that by analyzing all possible remainders of coordinates of $F$ modulo powers of $2$. First, assume $t$ is even, then at least one more number among $x,y,z$ is even, say $z$. Then midpoint of $BF$ has coordinates 
\begin{equation*}
\left(\frac{x+1}{2},\frac{y}{2},0,0\right) \pmod 1
\end{equation*}
and it is an integer translation of the midpoint of $AB$ (if $x$ and $y$ are odd) or $AC$ (if $x$ and $y$ are even). The only case that will not contradict that $\mathcal{T}$ is face-to-face is when $BF$ is parallel and equal to $AC$. Similarly we get that $CF$ is parallel and equal to $AB$, so $F=(1,1,0,0)$. Also we can get five more coordinate permutations of this point in the case $F$ has an even coordinate, in all other cases $x,y,z,t$ are odd.

We know that $x,y,z,t$ are odd and their sum is 2, so possible cases for modulo 4 remainders are $(3,3,3,1)$ and $(1,1,1,3)$ ($x,y,z,t$ are equivalent, so we will treat these cases as coordinates for $(x,y,z,t)$ modulo 4).
In the first case
\begin{equation*}
\frac{B + C + D + F}{4}=\frac{3A + E}{4} \pmod 1,
\end{equation*}
and the tiling is non face-to-face, so only the case $(1,1,1,3)$ of remainders modulo 4 is possible.

\begin{lemma}
For any $k\geq 2$ the remainders of the coordinates of $F$ modulo $2^k$ are $(1,1,a,2^k-a)$ for some odd $a\in[0,2^k-1]$, probably permuted.
\end{lemma}
\begin{proof}
We prove the statement by induction on $k$. The basis of induction is true for $k=2$ and $a=1$. Suppose the lemma is true for $k$ and we will prove it for $k+1$. All the coordinates in our proof could be permuted, and when we consider a coordinate modulo $n$ we usually take a representative from the interval $[0,n)$.

The point $F$ has coordinates $(1,1,a,2^k-a)$ modulo $2^k$ with odd $a<2^k$. Taking in account that sum of all coordinates is $2$ there are five options for the remainders of coordinates modulo $2^{k+1}$:

\begin{itemize}
\item $F=(1, 1,     a,     2^{k+1}-a) \pmod {2^{k+1}}$. This case satisfies requirements of the induction step.
\item $F=(1, 1,     2^k+a, 2^{k}-a) \pmod {2^{k+1}}$. This case satisfies requirements of the induction step. 
\item $F=(1, 2^k+1, a,     2^{k}-a) \pmod {2^{k+1}}$. One of numbers $a$ or $2^k-a$ is less than $2^{k-1}$, without loss of generality we can assume that $0<a<2^{k-1}$. Then the interval $(2^k,2^{k+1})$ contains at least two multiples of $a$, so there is a positive odd number $b<2^{k+1}$ such that $2^k<ab<2^{k+1}$. Then $b F=(b,2^k+b, ab, 2^{k+1}+2^k-ab)$ modulo $2^{k+1}$, and 
\begin{multline*}
  \frac{b F + (2^k-b)C+(2^{k+1}-ab)D + (ab-2^k)E}{2^{k+1}}=\\
  =\left(\frac{b}{2^{k+1}},0,0,0\right)=\frac{(2^{k+1}-b)A + b B}{2^{k+1}} \pmod 1
\end{multline*}
and the tiling is not face-to-face.
\item $F=(1,2^{k}+1,2^k+a,2^{k+1}-a) \pmod {2^{k+1}}$. Then $(2^k+1)F=(2^k+1,1,a,2^k-a)$ modulo $2^{k+1}$, and 
\begin{multline*}
  \frac{(2^k+1)F + (2^k-1)B}{2^{k+1}}=\\
  =\left(0,\frac{1}{2^{k+1}},\frac{a}{2^{k+1}},\frac{2^k-a}{2^{k+1}}\right) = \frac{(2^k-1)A + C + a D + (2^k-a)E}{2^{k+1}} \pmod 1
\end{multline*}
and the tiling is not face-to-face.
\item $F=(2^k+1,2^{k}+1,2^k+a,2^{k}-a) \pmod {2^{k+1}}$. Then $(2^k+1)F=(1,1,a,2^{k+1}-a)$ modulo $2^{k+1}$, and
\begin{multline*}
  \frac{(2^k+1)F + (2^k-a-1)D + a E}{2^{k+1}}=\\
  =\left(\frac{1}{2^{k+1}},\frac{1}{2^{k+1}},\frac{2^k-1}{2^{k+1}},0\right)=\frac{(2^k-1)A + B + C + (2^k-1)D}{2^{k+1}} \pmod 1
\end{multline*} 
and the tiling is not face-to-face.
\item $F=(2^k+1,2^{k}+1,a,2^{k+1}-a) \pmod {2^{k+1}}$. Then $(2^k+1)F=(1,1,2^k+a,2^{k}-a)$, and
\begin{multline*}
  \frac{(2^k+1)F + (2^k-a)D + (a-1)E}{2^{k+1}}=\\
  =\left(\frac{1}{2^{k+1}},\frac{1}{2^{k+1}},0,\frac{2^k-1}{2^{k+1}}\right)=\frac{(2^k-1)A+B+C+(2^k-1)E}{2^{k+1}} \pmod 1
\end{multline*}
and the tiling is not face-to-face.
\end{itemize}

Thus, the induction step is proved.
\end{proof}

We proceed with the proof of the theorem.

We can take $k$ such that $2^k>2\max(|x|,|y|,|z|,|t|)$, then the only possibility for coordinates with remainder 1 modulo $2^k$ is 1, so two coordinates of $F$ are 1's and two other add up to 0, so $F=(1,1,a,-a)$ for some positive odd number $a$ (or permutation). If $a\geq 3$, then 
\begin{equation*}
\frac{F + (a-1)B}{a}=\left(0,\frac{1}{a},0,0\right)=\frac{(a-1)A + B}{a} \pmod 1,
\end{equation*}
and if the tiling is face-to-face then the edge $BF$ is a translation of the edge $AB$. In that case $F=(0,0,0,2)$, so it doesn't have all odd coordinates. Therefore $a=1$ and $F=(1,1,1,-1)$.

In total we get 10 options for the point $F$: $(1,1,0,0)$ (all six permutations), and $(1,1,1,-1)$ (all four permutations).
\end{proof}

\begin{theorem}\label{SEC_Classification_dimension4}
(i) There are exactly four periodic triangulations of $\ZZ^4$ up to $\GL_4(\ZZ)$ equivalence.

(ii) Any partial triangulation of $\ZZ^4$ is extensible to a full triangulation of $\ZZ^4$.
  
\end{theorem}

\begin{proof}
We use Lemma \ref{thm:4dim-neigh} with exact classification of neighbors
for an exhaustive computer-assisted search.
We start from one simplex of volume $1$ and add adjacent simplices one by one by considering
all possibilities.
The number of cases to consider is kept down by keeping only non-isomorphic partial tilings in memory.
The software is available at \cite{DataPeriodicTrig} as a GAP package.
In the end we get four non-equivalent triangulations three of which are Delaunay triangulations
and the ``red-triangular'' triangulation \cite[Example 5.13.1]{Alexeev_semiabelian} which proves (i).

The intermediate object of the enumeration are exactly the partial triangulations of $\ZZ^4$.
It turns out that in the enumeration it never happenned that a partial triangulation had no extensions
by adding simplex which proves (ii).
\end{proof}

\section{Local approach in higher dimensions}\label{sec:dim5}

In this section we show that local approach we used in Lemma \ref{lem:3dim-neigh}
and Theorem \ref{thm:4dim-neigh} can not prove finiteness of non-equivalent
triangulations in dimension at least $5$.

\begin{theorem}\label{THM_infinite_sequence_intersection}
For $n\geq 5$ there exist a simplex $S$ of volume $1$ and an infinite sequence $S_k$
of simplices of volume $1$ such that $S\cap S_k$ is a facet and the translates of
$S$ and $S_k$ are not intersecting.
\end{theorem}
\proof We first consider the case $n=5$.

We fix simplex $S=OABCDX$ where $O=(0,0,0,0,0)$, $X=(-1,0,0,0,0)$, $A=(0,1,0,0,0)$,
$B=(0,0,1,0,0)$, $C=(0,0,0,1,0)$, and $D=(0,0,0,0,1)$. We will show that there are
infinitely many options to choose a neighbor $T$ of $S$ adjacent by the facet $x_1=0$
such that $\ZZ^5$ translations of $S$ and $T$ do not violate the face-to-face property. 

Let $X'=(1,1,1,1,k+1)$ for any $k\geq 0$, then $T=OABCDX'$ will satisfy this condition.

For any $n$ both simplices $S$ and $T$ have volume 1, so $S$ doesn't intersect translations of
$S$, and $T$ doesn't intersect translations of $T$. It is enough to show
that an arbitrary integer translation of $S$ doesn't intersect $T$ other than by
vertices or by the facet $x_1=0$.

Consider the translation $S'$ of $S$ by the integer vector $(a,b,c,d,e)$.
Assume the intersection $S'\cap T$ contains a point $\mathbf{x}$ which is not a vertex
of $T$ and has non-zero first coordinate.
Since $S$ satisfies the inequality $-1\leq x_1\leq 0$ and $T$ satisfies the inequality
$0\leq x_1\leq 1$ we must have $a=1$.
Since $\mathbf{x}$ is a point of intersection of $S'$ and $T$, then $\mathbf{x}$
is in the cone with vertex $(0,b,c,d,e)$ generated by the vectors $(1,0,0,0,0)$, $(1,1,0,0,0)$,
$(1,0,1,0,0)$, $(1,0,0,1,0)$, and $(1,0,0,0,1)$, the edges of $S$ from the vertex $X$.

Then the point $(0,b,c,d,e)$ is in the cone with vertex $\mathbf{x}$ generated by negatives
of these vectors. Since $\mathbf{x}$ is in $T$, so $(0,b,c,d,e)$ is in the convex hull
of the $6$ cones with vertices at vertices of $T$ generated by the vectors $(-1,0,0,0,0)$,
$(-1,-1,0,0,0)$, $(-1,0,-1,0,0)$, $(-1,0,0,-1,0)$, and $(-1,0,0,0,-1)$.
We are interested only in the part of the convex hull of these 6 cones in the plane $x_1=0$,
and the extremal points of these cones are 5 vertices of $T$
(except $(1,1,1,1,k+1)$) and the points $(0,1,1,1,k+1)$, $(0,0,1,1,k+1)$, $(0,1,0,1,k+1)$,
$(0,1,1,0,k+1)$, and $(0,1,1,1,k)$. These are $5$ points of intersection of the edges of the
cone with the vertex at $(1,1,1,1,k+1)$ with the hyperplane $x_1=0$.

Thus the point $(0,b,c,d,e)$ is in the convex hull of the $10$ points $(0,0,0,0,0)$,
$(0,1,0,0,0)$, $(0,0,1,0,0)$, $(0,0,0,1,0)$, $(0,0,0,0,1)$, $(0,1,1,1,n+1)$, $(0,0,1,1,n+1)$,
$(0,1,0,1,n+1)$, $(0,1,1,0,n+1)$, and $(0,1,1,1,n)$.

Now we can see that $b,c,d$ must be 0's or 1's. Since the convex hull is
centrally symmetric with respect to the point $\left(0,\frac12,\frac12,\frac12,\frac{k+1}{2}\right)$
we can assume that $b=c=0$.
Then the point $(0,b,c,d,e)$ must be in the convex hull of only three points $(0,0,0,0,0)$,
$(0,0,0,1,0)$, and $(0,0,0,0,1)$, and there is no such point except themselves.
None of these points is in the interior of the convex hull of the 10 points above,
so we have found a contradiction with existence of such a point $\mathbf{x}$.

For $n > 5$ the idea is simply to take the pair of simplices and simply add another
point. \qed


\section{Flipping and five-dimensional partial enumeration}\label{sect:flip}

Let us consider a periodic triangulation $\mathcal{T}$ of $\ZZ^n$. Given a simplex $S$ and a facet $F$ of $S$,
we can consider the adjacent simplex $S(F)$ to $S$. The union of the vertex sets of $S$ and $S(F)$
is a set of $n+2$ points and we call the convex hull of those $Cv(S,F)$.

\begin{lemma}\label{LEM_repartitioning_polytope}
Any $n$-dimensional convex polytope with $n+2$ vertices (called {\em repartitioning polytope}) admits exactly $2$ triangulations.
\end{lemma}
\begin{proof}
Suppose that the vertices are $v_1$, \dots, $v_{n+2}$ then there is exactly one linear relation of the form
\begin{equation*}
a_1 v_1 + \dots + a_{n+2} v_{n+2} = 0
\end{equation*}
up to a non-zero multiple. 
For $1\leq j\leq n+2$ we define $S_j$ the simplex formed by $v_i$ for $i\in \{1,\dots,n+2\} - \{ j\}$.
The first triangulation is formed by the simplices $S_j$ for $j$ such that $a_j \geq 0$ and the second
triangulation by the simplices $S_j$ for $j$ such that $a_j \leq 0$.
\end{proof}

Suppose now that the simplices of $T$ contained in $Cv(S,F)$ have determined a tiling of it.
Then the simplices in $Cv(S,F)$ form a triangulation and we can swap it into another triangulation.
Unfortunately things are not always so simple. If we have $a_j=0$ for some $j$ then the vertices $v_i$
for $i\in \{1,\dots,n+2\} - \{j\}$ define a $(n-1)$-dimensional polytope with $n+1$ vertices.
Therefore flipping the triangulation of $Cv(S,F)$ also flips the triangulation of the facets of
$Cv(S,F)$. The set $Irr(S,F)= \tconv\{v_i | a_i\not= 0\}$ defines a face of $Cv(S,F)$.
Thus if one flips the triangulation in $Cv(S,F)$, then one needs to flip it in all repartitioning
polytopes containing $Irr(S,F)$ as well. We call a family of such flips {\it coherent}.
This kind of flip is sometimes called {\em bistellar flip} in the literature.

Note that in the case of Delaunay triangulations the flips that are considered are formed by several
bistellar flips done at the same time.

\begin{theorem}\label{THM_950_triangulation}
There are at least $950$ periodic triangulations of $\ZZ^5$ up to $\GL_5(\ZZ)$ equivalence.
\end{theorem}
\begin{proof} 
Given a periodic triangulation of $\ZZ^5$ we consider all ways
to do a coherent flipping on it.
We thus obtain a set of new periodic triangulations.
We insert element of this list into the list of known periodic triangulations
if they are not isomorphic to a triangulation already known.
We start from one arbitrary Delaunay triangulation of $\ZZ^5$.
We finish when all periodic triangulations in the list have been treated.
Since the finiteness of the set of periodic triangulation is not proved in
dimension $5$ this process was not guaranteed to terminate.
But it did and yielded $950$ periodic triangulations.
The code is available at \cite{DataPeriodicTrig}. 
\end{proof}

The list of $950$ periodic triangulations ($222$ of them Delaunay) is interesting in its own right
and is available at \cite{DataPeriodicTrig}.
The volumes of the simplices in the list of $950$ triangulations are $1$ or $2$ which corresponds to the possible volume of simplices in Delaunay tessellations.
Given a simplex $S$ of volume $1$ and vertices $v_0$, \dots, $v_5$ we can consider which simplices $S'$ can be adjacent to $S$. Their vertex set will be of the form
\begin{equation*}
\{w\}\cup \{v_j\}_{0\leq j\leq 5, j\not= 5} \mbox{~with~} w = \sum_{j=0}^5 b_j v_j \mbox{~and~} 1 = \sum_{j=0}^5 b_j
\end{equation*}
Thus we can encode them by a pair $\{(b_0, \dots, b_5), i\}$. Up to permutation with the list of $950$ possible tilings we found following possibilities for the pairs:
\begin{equation*}
\begin{array}{ccc}
\{(-1,1,1,0,0,0), 0\}       &      \{(-1,-1,1,1,1,0), 0\}     &    \{(-1,-1,-1,1,1,2), 0\}\\
\{(-2,-1,1,1,1,1), 1\}      &      \{(-1,-1,-1,-1,2,3), 0\}    
\end{array}
\end{equation*}
The last possibility $\{(-1,-1,-1,-1,2,3),0\}$ does not show up in the case of Delaunay triangulations.

The symmetry of the tiling varies widely with one of the periodic tiling having a point
group symmetry isomorphic to the symmetric group $\Sym(6)$.

\begin{theorem}\label{THM_notdelaunay_centrallysymmetric}
Periodic triangulations of $\ZZ^n$ which are not Delaunay but are centrally symmetric exist for $n\geq 5$.
\end{theorem}
\begin{proof}
For $n=5$ it suffices to take one of the $23$ triangulations out of $950$ known in dimension $5$
that are not Delaunay but are centrally symmetric.
For $n>5$ this tiling ${\mathcal T}$ can be extended with tiles of the form $\Delta \times [0,1]^{n-5}$
for $\Delta$ a $5$-dimensional simplex of ${\mathcal T}$.
By applying Lemma \ref{Refinement_Theorem} (iii) for an arbitrary generic quadratic form
we obtain a $\ZZ^n$-periodic triangulation.
This triangulation is centrally symmetric since $x\mapsto -x$ is a symmetry of the original
tiling but also of the quadratic form.
\end{proof}

Note that existence of a periodic centrally
symmetric non-Delaunay triangulation for $n=8$ was established in \cite{LocalCoveringOptimality}.

\begin{theorem}\label{THM_nonregular_triangulation}
There exist non-regular periodic triangulations for $n\geq 5$.
\end{theorem}
\begin{proof}
For $n=5$ we apply the method of subsection \ref{Sec_Comput_Tools_-_Test_Regularity}
to one of the $950$ triangulations of Theorem \ref{THM_950_triangulation}.
The list of $3264$ simplices of the triangulation number 430 that cannot be part
of a regular triangulation is available at \cite{DataPeriodicTrig}.
For $n>5$ this tiling ${\mathcal T}$ can be extended with tiles of the
form $\Delta \times [0,1]^{n-5}$ with $\Delta$ a $5$-dimensional simplex of ${\mathcal T}$.
By applying Lemma \ref{Refinement_Theorem} (iii) for an arbitrary generic quadratic form
we obtain a $\ZZ^n$-periodic triangulation which is necessarily non-regular.
\end{proof}

\section{Open problems}\label{sec:open_questions}
In this section we list a number of interesting questions that showed up in the course of this
research.

\subsection{Finiteness and enumeration}
A natural question that we were unable to resolve is whether there are finitely many
$\ZZ^n$-periodic triangulations of $\ZZ^n$ up to the action of $\GL_n(\ZZ)$?
Theorem \ref{THM_infinite_sequence_intersection} shows that a local approach considering only
pairs of simplices will not work.

There are many related question. For example in a fixed dimension $n$, is the set of all periodic
triangulations of $\ZZ^n$ connected by flipping? The resolution of such questions is certainly
very hard since analogue questions about triangulations of the hypercube are still unsolved
\cite{DeLoeraRambauSantos}.
The resolution of the above connectedness would imply that the number of triangulations in
dimension $5$ is exactly $950$.

A proof of finiteness in dimension $5$ would not a priori give an algorithm for
the enumeration since we do not know the possible volume of simplices nor the
adjacencies between them.



\subsection{Extensibility of partial triangulations}
In a lot of contexts of this search we reach a point where we had a partial triangulation
of $\ZZ^n$ and we wanted to extend it to a full triangulation. Is this always possible?
If so what would be a process for obtaining such a triangulation? If this extensibility
were true then we would have an infinity of types of periodic triangulation in dimension $5$.
Note that Theorem \ref{SEC_Classification_dimension4} proves that this extensibility
holds in dimension $n\leq 4$.

One possible way to consider the problem would be following \cite{Chew_ConstrainedDelaunay}
to consider {\em constrained Delaunay triangulations} and see if the relevant notion
could be extended to our case. It would require a twofold generalization: a generalization from
dimension $2$ to any dimension and a generalization to the periodic case.

\subsection{Regularity}

Is every periodic regular triangulations also Delaunay? The answer is not known.
As we saw in Section \ref{Sec_Comput_Tools} we can test regularity on finite subsets
of $\ZZ^n$ by linear programming. 
But we need actually to define the height function all over $\ZZ^n$.
Finding such explicit function is difficult since as soon as we impose some
translational invariance on the function $f$ we obtain a function that is actually
quadratic.


Is the ``red-triangular'' \cite[Example 5.13.1]{Alexeev_semiabelian} $\ZZ^4$-periodic
triangulation regular?
If this triangulation is restricted to a set of $12864$ simplices containing $1224$
points then we can found a corresponding function $f$ which indicates that this
triangulation is likely to be regular.


\subsection{Volume of simplices}

What is the maximum volume of a simplex in a periodic triangulation? So far in all cases
considered, we found that the volumes of the simplices occurring was not higher than
the volume of the simplices of the Delaunay triangulations in the same dimension
which are $1$, $2$, $3$ and $5$, respectively in dimension $n\leq 4$, $5$, $6$ and $7$,
respectively \cite{InhomogeneousPerfect}.
We see no reason why this should always be the case.


%


\section{Acknowledgments}

We thank Francisco Santos and Achill Sch\"urmann for interesting discussions on this work.

\bibliographystyle{amsplain_initials_eprint}
\bibliography{LatticeRef}

\end{document}